\documentclass[11pt]{article}

\usepackage[a4paper, margin={0.7in}]{geometry}
\usepackage[utf8]{inputenc}
\usepackage{lipsum}
\usepackage[all]{xy}
\usepackage{tikz-cd}
\usepackage{color}

\usepackage[T1]{fontenc}
\usepackage{tgpagella}
\usepackage{amssymb}
\usepackage{amsfonts}
\usepackage{amsmath}
\usepackage{amsthm}
\usepackage{amscd}
\usepackage{bussproofs}
\usepackage[all]{xy}
\usepackage{comment}
\usepackage{tikz-cd}
\usepackage{amssymb}
\usepackage{stmaryrd}
\usepackage{amsfonts}
\usepackage{amsmath}
\usepackage{enumitem}
\usepackage{color,mathdots,setspace,framed,graphicx,natbib,latexsym,nicefrac,longtable}
\usepackage[textwidth=3.5cm]{todonotes}

\newtheorem{tw}{Theorem}

\newtheorem{lemat}[tw]{Lemma}

\newtheorem{claim}[tw]{Claim}
\theoremstyle{definition}

\newtheorem{przyklad}[tw]{Example}
\newtheorem{uwaga}[tw]{Remark}

\newcommand{\PA}{\textnormal{\textsf{PA}}}

\newcommand{\md}[1]{\mathcal{#1}}

\newcommand{\M}{\md{M}}

\title{Corrigendum \& Addendum to: 

\emph{Categoricity-like Properties in the First Order Realm}}
\author{Ali Enayat\thanks{Department of Philosophy, Linguistics, and Theory of Science, University of Gothenburg, Sweden \texttt{(e-mail:~ali.enayat@gu.se)}. Research supported by the National Science Centre, Poland (NCN), grant number 2019/34/A/HS1/00399.}      
    \and  
    Mateusz \L{}e\l{}yk\thanks{Faculty of Philosophy, University of Warsaw, Poland (\texttt{e-mail:~mlelyk@uw.edu.pl}). Research supported by the National Science Centre, Poland (NCN), grant number 2022/46/E/HS1/00452.}
    }

\begin{document}

\maketitle

\section{Introduction}This note complements our paper \cite{Enayat_Łełyk_2024}, which we assume the reader has at hand for ready reference. Sections 2 and 3 constitute the "Corrigendum'' component; the former repairs the proof of Theorem 39 of \cite{Enayat_Łełyk_2024}, while the latter presents a counterexample to one of the lemmata used in the proof of Theorem 77 of \cite{Enayat_Łełyk_2024}. Section 3 is the "Addendum" component, in which we draw attention to relevant recent advances. 


\section{Correction to Theorem 39 of \cite{Enayat_Łełyk_2024}}

The statement of Theorem 39 of \cite{Enayat_Łełyk_2024} is correct, but the proof presented for it suffers from two defects. The more serious defect is that a stronger form of Lemma 40 of \cite{Enayat_Łełyk_2024} is needed for the proof of the theorem to go through; this stronger result is Theorem 3 below; whose part (b) is crucial for the proof of Theorem 39 of \cite{Enayat_Łełyk_2024}. The other defect, detected by Zachiri McKenzie, concerns the quantifier complexity of the axioms of the version of $\mathsf{KP}$ (Kripke-Platek set theory) used in the proof: The correct complexity is $\Pi_3$, and not $\Pi_2$.\footnote{With the exception of $\mathsf{Coll}(\Delta_0)$ (which is equivalent to $\mathsf{Coll}(\Sigma_1)$ in the presence of Kuratowski pairing), all the other axioms of $\mathsf{KP}$ can be expressed as $\Pi_2$-sentences; however, as shown by a recent result of McKenzie \cite[Theorem 4.9]{McKenzie-Kaufmann-Clote}, $\mathsf{Coll}(\Delta_0)$ cannot even be expressed as set of $\Sigma_3$-sentences. Note that one can use a $\Pi_3$-satisfaction predicate for $\Pi_3$-formulae to show that $\mathsf{KP}$ is axiomatized by a single $\Pi_3$-sentence, but we do not need this refinement here.} For the benefit of the reader, the entire corrected proof of the theorem is presented below.

\begin{tw}\label{Failure_of_tightness_for_bounded_ZF} If $\mathsf{ZF}$ is consistent, then for each $n\in\omega$ the following hold:
 \begin{enumerate}
 \item [(a)]$\mathsf{ZF}_{\Pi_n}$ is not solid.
 \item [(b)] $\mathsf{ZF}_{\Pi_n}$ is not tight.
 \end{enumerate}
\end{tw}
\begin{proof} Our proof strategy is similar to the one for Theorem 25 of \cite{Enayat_Łełyk_2024}. Assuming the consistency of $\mathsf{ZF}$, by G\"{o}del's second incompleteness theorem, and G\"{o}del's relative consistency proof of $\mathsf{ZF + V = L}$, the theory $\mathsf{ZF+V=L+\lnot Con_{ZF}}$ is consistent, and therefore by the arithmetized completeness theorem there is a countable model $\mathcal{M}_0$ such that:
\begin{enumerate}
\item[(1)] $\mathcal{M}_0\models \mathsf{ZF+V=L+\lnot Con_{ZF}}$, and
\item[(2)] The elementary diagram of $\mathcal{M}_0$ is definable in $\mathbb{N}$ (the standard model of arithmetic); in particular, there is an interpretation $\mathcal{I}$ such that $\mathbb{N}\trianglerighteq^{\mathcal{I}}\mathcal{M}_0.$
\end{enumerate}
\noindent As in the arithmetical case, given $n \in \omega$, let $K_{n}(\mathcal{M}_0)$ be the submodel of $\mathcal{M}_0$ whose universe consists of elements of $\mathcal{M}_0$ that are definable in $\mathcal{M}_0$ by a $\Sigma_{n}$-formula. At the heart of the proof of Theorem 1 is the following result, which is the set-theoretical analogue of Theorem 24 of \cite{Enayat_Łełyk_2024}.

\begin{tw}
    
\label{N_and_K_n_are_biint_for_ZF+V=L} Let $\M_0$ and $K_{n}(\mathcal{M}_0)$ be as above. The following statements hold for each $n\geq 3$:
\begin{enumerate}
\item [(a)] $\mathbb{N}$ is bi-interpretable with $K_{n}(\mathcal{M}_0).$

\item [(b)] $\mathrm{Th}(\mathbb{N})$\ is bi-interpretable with $\mathrm{Th}(K_{n}(\mathcal{M}_0)).$
\end{enumerate}
\end{tw}

\noindent The proof of part (a) of Theorem \ref{N_and_K_n_are_biint_for_ZF+V=L} is based on Theorem  \ref{Set_theoretical_analogue_of_Paris_Kirby} below, which is the set-theoretical analogue of a basic result in the model theory of arithmetic (Theorem 22 of \cite{Enayat_Łełyk_2024}). The proof of this set-theoretical analogue takes more effort to establish since in the arithmetical case the well-ordering $<$ of the universe has a $\Delta_0$-graph, but in our context we need to rely on technical features of the well-ordering $<_{\mathsf{L}}$ of the constructible universe $\mathsf{L}$. 

\noindent In part (c) of the theorem below, $\mathsf{Coll}(\Sigma_{n+1})$ is the collection scheme for $\Sigma_{n+1}$-formulae of set theory, whose instances are of the form:
$$\left( \forall x\in a\ \exists y\ \sigma(x,y)\right) \rightarrow \left(
\exists b\ \forall x\in a\ \exists y \in b ~\sigma(x,y)\right),$$ where $\sigma$ is a $\Sigma_{n+1}$-formula whose parameters are suppressed. Also note that in the statement of the theorem below, $\mathsf{ZF^-}$ denotes $\mathsf{ZF\setminus\{Powerset\}}$. It is well-known that the powerset axiom is not needed for the development of the basic properties of the constructible universe (see, e.g., Section 5 of Chapter II of \cite{Barwise_book}).


\begin{tw} \label{Set_theoretical_analogue_of_Paris_Kirby} The following holds for any model $\mathcal{M}\models \mathsf{ZF^{-} + V = L}$, where $n\geq 3$.\footnote{As pointed out by Zachiri McKenzie, in Lemmas 5.15 and 5.19 of \cite{Mathias-maclane}, Mathias established an \emph{internal version} of Theorem 3 within models of certain fragments of $\mathsf{ZF+V=L}$.}

\begin{enumerate}

\item [(a)] $K_{n}(\mathcal{M})\prec _{\Pi _{n}}\mathcal{M}$, hence $K_{n}(\mathcal{M}%
)\models \mathrm{Th}_{\Pi _{n+1}}(\mathcal{M}).$

\item [(b)] $K_n(K_{n}(\mathcal{M})) = K_{n}(\mathcal{M})$, i.e., every element of $K_{n}(\mathcal{M})$ is definable in $K_{n}(\mathcal{M})$ by some $\Sigma _{n}$-formula.\footnote{Note that in contrast with Theorem 22 of \cite{Enayat_Łełyk_2024}, here there is no nonstandardness assumption.}

\item [(c)]  $K_{n}(\mathcal{M})\models \mathsf{ZF}_{\Pi _{n+1}}
+\lnot \mathsf{Coll}(\Sigma_{n+1}).$

\end{enumerate}
\end{tw}
 \begin{proof}[Proof of Theorem \ref{Set_theoretical_analogue_of_Paris_Kirby}] Let $n\geq 3$, and $\mathcal{M}\models \mathsf{ZF^-}+%
\mathsf{V}=\mathsf{L}$. Also, let $<_{\mathsf{L}}$ denote the usual
well-ordering of the constructible universe, which is well-known to be
definable by a $\Sigma _{1}$-formula within the theory $\mathsf{ZF^-}+\mathsf{V%
}=\mathsf{L}$. To show that $K_{n}(\mathcal{M})\prec _{\Sigma _{n}}\mathcal{M%
}$, by the Tarski-test for elementarity, and the availability of a $\Delta
_{0}$-definable pairing function (such as Kuratowski's) for coding finite
tuples of elements as a single element, it suffices to show:\medskip

\noindent Claim $(\ast )$. If $c\in K_{n}(\mathcal{M})$, and $\mathcal{M}%
\models \exists x\ \sigma (x,c)$ for some $\Sigma _{n}$-formula $\sigma
(x,y) $, then there is some $d\in K_{n}(\mathcal{M})$ such that $\mathcal{M}%
\models \sigma (c,d).$ \medskip

\noindent To verify Claim $(\ast )$, suppose that $c\in K_{n}(\mathcal{M})$, and $%
\mathcal{M}\models \exists x\ \sigma (x,c)$, where $\sigma (x,c)=\exists y\
\pi (x,y,c)$ and $\pi $ is a $\Pi _{n-1}$-formula. Thus we have: 
\begin{equation*}
\mathcal{M}\models \exists x\ \exists y\ \pi (x,y,c).
\end{equation*}%
Using the Kuratowski pairing function $\left\langle x,y\right\rangle=\{\{x\},\{x,y\}\} $, we
can write the above as: 
\begin{equation*}
\mathcal{M}\models \exists z\ \pi (\left( z\right) _{0},\left( z\right)
_{1},c)\text{,}
\end{equation*}%
where $z=\left\langle \left( z\right) _{0},\left( z\right) _{1}\right\rangle 
$. By invoking the fact that $<_{\mathsf{L}}$ well-orders the universe
within $\mathcal{M}$, we can conclude that for some $d\in M$, $\mathcal{%
M}\models \psi (d,c)\wedge \exists !z\ \psi (z,c)$, where: 
\begin{equation*}
\psi (z,a):=\overset{\psi _{1}(z,c)}{\overbrace{\pi (\left( z\right)
_{1},\left( z\right) _{2},c)}}\wedge \overset{\psi _{2}(z,c)}{\overbrace{%
\forall u<_{\mathsf{L}}z\ \lnot \pi (\left( u\right) _{1},\left( u\right)
_{2},c)}}.
\end{equation*}

\begin{itemize}
\item We will show that $d$ is $\Sigma_n$-definable in the structure $(\mathcal{M},c)$ by showing that $\psi(z,c)$ is equivalent within $\mathcal{M}$ to a $\Sigma _{n}$-formula $\theta(z,c)$ with parameter $c$. This is sufficient to prove Claim $(\ast )$ since we can then take advantage of the assumption that $c\in K_{n}(\mathcal{M})$ to eliminate the parameter $c$ to obtain a paramater-free $\Sigma _{n}$-formula $\psi ^{\ast }(z)$ that defines $d$ in $\mathcal{M}$. More specifically:%
\begin{equation*}
\psi ^{\ast }(z):=\exists v\left[ \theta (v)\wedge \psi (z,v)\right] ,
\end{equation*}%
where $\theta (v)$ is a $\Sigma _{n}$-formula that defines $c$ in $\mathcal{M%
}$. \medskip
\end{itemize}

\noindent The first conjunct $\psi _{1}(z,c)$ of $\psi (z,c)$ is equivalent to 
\begin{equation*}
\exists w_{1}\ \exists w_{2}\ \left[ w=\left( z\right) _{1}\wedge w=\left(
z\right) _{2}\wedge \pi (\left( z\right) _{1},\left( z\right) _{2},c)\right]
,
\end{equation*}%
since $w=\left( z\right) _{i}$ is a $\Delta _{0}$-formula for $i\in \{0,1\}.$ By assumption $\pi $ is a $\Pi _{n-1}$-formula, thus
$\psi _{1}(z,c)$ is equivalent in $\mathcal{M}$ to
a $\Sigma _{n}$-formula. Thus the proof will be complete once we verify that
the second conjunct $\psi _{2}(z,c)$ of $\psi (z,c)$ is also equivalent to a 
$\Sigma _{n}$-formula in $\mathcal{M}$ (with parameter $c$).\footnote{%
Note that $\psi _{2}(z,c)$ is readily expressible as $\Pi _{n}$-formula
(when $n\geq 3$) since $<_{\mathsf{L}}$ has a $\Sigma _{1}$-definition, but
this fact does not help us here.} \ Indeed, we will do better by showing that $\psi (z,c)$
is equivalent to a $\Sigma _{n-1}$-formula in $\mathcal{M}$ (with parameter $%
c$). For this purpose, we need to revisit the specific way that the
canonical well-ordering $<_{\mathsf{L}}$ of the constructible universe $%
\mathrm{L}$ was designed by G\"{o}del. The following two facts are
well-known, see pp.~188-190 of Jech's textbook \cite{Jech_Set_Theory} for an exposition:

\begin{enumerate}
\item[(1)] $\mathsf{L:}=\bigcup\limits_{\alpha \in \mathsf{Ord}}L_{\alpha }$%
, and $<_{\mathsf{L}}\ :=\bigcup\limits_{\alpha \in \mathsf{Ord}}<_{\alpha }$,
where $<_{\alpha }$ well-orders $L_{\alpha }$ for each $\alpha $, and the
maps $$\alpha \mapsto L_{\alpha } \mathrm{ ~and~ } \alpha \mapsto <_{\alpha }$$ are both $%
\Delta _{1}^{\mathsf{ZF^-}}$.

\item[(2)] If $\alpha \in \beta \in \mathsf{Ord}$, then $<_{\alpha }$ is
\emph{end extended} by $<_{\beta }$, i.e., 
\begin{enumerate}
\item[$(i)$] $(x<_{\alpha }y)\rightarrow $ $%
(x<_{\beta }y)$; and \item[$(ii)$] $(x\in L_{\alpha })\wedge (y\in L_{\beta }\setminus
L_{\alpha })\rightarrow (x<_{\beta }y)$. 
\end{enumerate}Note that this implies that if $%
z\in L_{\alpha }$ and $u<_{\mathsf{L}}z$, then $u\in L_{\alpha }.$
\end{enumerate}

\noindent Facts (1) and (2) imply that $\mathcal{M}\models \forall u~ [(u<_{\mathsf{L}}z)\leftrightarrow (u<_{\alpha}z)]$, where $u<_{\alpha}z$ is shorthand for $\left\langle u,z\right\rangle \in~ <_{\alpha}.$ This makes it clear that $\psi _{2}(z,c)$ is
equivalent to the following formula:

\begin{equation*}
\exists \alpha \ \exists r\ \exists w\ \left[ \mathsf{Ord}(\alpha )\wedge
\left( r=\ <_{\alpha }\right) \wedge \left( w=L_{\alpha }\right) \wedge
\left( z\in w\right) \wedge \forall u\in w\ \left( (\left\langle
u,z\right\rangle \in r\rightarrow \lnot \pi (u,c)\right) \right] .
\end{equation*}%
By inspection, the above is $\Sigma _{n-1}$ (in
parameter c), as shown by the following calculation:\medskip 

\begin{center}
$\overset{\Sigma _{n-1}\text{\ }(n\geq 3)}{\overbrace{\exists \alpha \
\exists r\ \exists w\ \left[ \overset{\Delta _{0}}{\overbrace{\mathsf{Ord}%
(\alpha )}}\wedge \ \overset{\Delta _{1}}{\overbrace{\left( r=\ <_{\alpha
}\right) }}\wedge \ \overset{\Delta _{1}}{\overbrace{\left( w=L_{\alpha
}\right) }}\wedge \overset{\Delta _{0}}{\overbrace{\left( z\in w\right) }}%
\wedge \overset{\Sigma _{n-1}\text{\ }(\mathrm{by\ }\mathrm{Collection,~}n\geq 3)}{\overbrace{\forall u\in w\ \overset{%
\Sigma _{n-1}\text{\ }(n\geq 3)}{\overbrace{\left( (\overset{\Sigma _{1}}{%
\overbrace{\left\langle u,z\right\rangle \in r}}\rightarrow \overset{\Sigma
_{n-1}}{\overbrace{\lnot \pi (u,c)}}\right) }}}}\right] }}.$
\end{center}

\noindent This concludes the proof of part (a) of Theorem \ref{Set_theoretical_analogue_of_Paris_Kirby}. Note that part (b) is an immediate consequence of part (a). So we now turn to the proof of (c). 

In order to demonstrate the failure of  $\mathsf{Coll}(\Sigma_{n+1})$ in $K_{n}(\mathcal{M})$ we first we need to address the question of definability of $\Sigma_k$-satisfaction within \textit{fragments} of $\mathsf{ZF}$ (it has been known since Levy's pioneering work that within full $\mathsf{ZF}$ there is a $\Sigma_k$-satisfaction class that is $\Sigma_k$-definable for each $k\geq 1$).  As readily seen by examining Definitions 2.9 and 2.10 of McKenzie's \cite{McKenzie-collection}, within $\mathsf{KP}$, there is a $\Sigma_k$-satisfaction class that is $\Sigma_k$-definable for each $k\geq 1$; more detail for this construction is provided by Mathias in \cite[p.~156]{Mathias-maclane}. Here $\mathsf{KP}$ denotes Kripke-Platek set theory. In our formulation (following recent practice, led by Mathias \cite{Mathias-maclane}) the scheme of foundation is limited to $\Pi _{1}$-formulae (equivalently: the scheme of $\in $-induction for $\Sigma _{1}$-formulae). Thus in contrast to Barwise's $\mathsf{KP}$ in 
\cite{Barwise_book}, which includes the full scheme of foundation, our
version of $\mathsf{KP}$ is axiomatized by a set of  $\Pi_3$-sentences (as noted earlier).


With the above preliminaries in place, part (a) of Theorem 23 and the $\Pi_3$-axiomatizability of $\mathsf{KP}$ make it clear that for $n \geq 3$, $K_{n}(\mathcal{M})\models \mathsf{KP}$. This assures us that there is a $\Sigma_{n}$-satisfaction predicate in $K_{n}(\mathcal{M})$ for $\Sigma_{n}$-formulae. Consider the function $f$ that maps (the code of) each $\Sigma_{n}$-formula in $K_{n}(\mathcal{M})$ to $0$, if there is no element that satisfies $\sigma(x)$, and otherwise to the least $<_{\mathsf{L}}$-element satisfying $\sigma(x)$.  Thanks to the availability of a $\Sigma_{n}$-definable satisfaction predicate in $K_{n}(\mathcal{M})$ for $\Sigma_{n}$-formula, and the $\Sigma_1$-definability of $<_{\mathsf{L}}$, the graph of $f$ is readily seen to be defined by a $\Sigma_{n+1}$-formula. Note that the domain of $f$ is a set in $K_{n}(\mathcal{M})$ but its range is the whole of $K_{n}(\mathcal{M})$. Thus $\mathsf{Coll}(\Sigma_{n+1}$)-fails in $K_{n}(\mathcal{M})$. This concludes the proof of part (c) Theorem \ref{Set_theoretical_analogue_of_Paris_Kirby}.
\end{proof}

At the very beginning of the proof of Theorem \ref{Failure_of_tightness_for_bounded_ZF}, we saw that our model $\M_{0}$ of $\mathsf{ZF+V=L}$ not only is interpretable in $\mathbb{N}$, but also has the stronger property that the elementary diagram of $\M_{0}$ is arithmetical. Hence the model $K_n(\M_{0})$ is also interpretable in $\mathbb{N}$. The next result shows that, conversely, $\mathbb{N}$ is interpretable in $K_n(\M_0)$ when $n\geq 3$.
\begin{lemat}\label{key_lemma} The standard cut $\omega$ is definable in $K_n(\mathcal{M}_0)$ for $n\geq 3$.
\end{lemat}

\begin{proof} Note that $K_n(\mathcal{M}_0)$ is $\omega$-nonstandard since $\M_0$ satisfies $\lnot \mathsf{Con_{ZF}}$, and the least proof of inconsistency of $\mathsf{ZF}$ in $\M$ has a $\Sigma_1$-definition and thus belongs to $K_n(\mathcal{M}_0)$. So, similar to the proof of Theorem 23 of \cite{Enayat_Łełyk_2024}, and relying on part (b) of Theorem \ref{Set_theoretical_analogue_of_Paris_Kirby}, we have:
    $i$ is a nonstandard member of $\omega^ {K_{n}(\mathcal{M}_0)}$ iff every element of $K_{n}(\mathcal{M}_0)$ is definable by a $\Sigma_n$-formula below $i$. This is first order expressible, thanks to the definability of a $\Sigma_{n}$-satisfaction predicate in $K_{n}(\mathcal{M}_0)$. 
\end{proof}

\begin{proof}[Proof of Theorem \ref{N_and_K_n_are_biint_for_ZF+V=L}]
As noted earlier, we have access to an interpretation $\mathcal{I}$ such that $\mathbb{N}~{\trianglelefteq^{\mathcal{I}}} ~\mathcal{M}_0$. In the other direction, Lemma \ref{key_lemma} gives us an interpretation $\mathcal{J}$ such that $\mathcal{M}_0~{\trianglelefteq^{\mathcal{J}}} ~\mathbb{N}$. An identical reasoning as the one used in the proof of part (a) of Theorem 24 of \cite{Enayat_Łełyk_2024} shows that the interpretations $\mathcal{I}$ and $\mathcal{J}$ witness the bi-interpretability of $\mathbb{N}$ and $K_n(\mathcal{M})$.

Part (b) of Theorem \ref{N_and_K_n_are_biint_for_ZF+V=L} then follows since the proof of bi-interpretability in part (a) is uniform.
 
 \end{proof}
With Theorem \ref{N_and_K_n_are_biint_for_ZF+V=L} at hand, we can finally wrap up the proof of Theorem \ref{Failure_of_tightness_for_bounded_ZF}. By Theorem \ref{Set_theoretical_analogue_of_Paris_Kirby}, given $2\leq n<k$, $\mathbb{N}$ is bi-interpretable with a model of $\mathsf{ZF}_{\Pi _{n}}+\lnot \mathsf{Coll}
(\Sigma _{n+1})$, and also $\mathbb{N}$ is bi-interpretable with a model of $\mathsf{ZF}_{\Pi _{k}}$. It is easy to see that $ \mathsf{Coll}(\Sigma_{n})$ is of complexity at most $\Pi_{n+3}$ for all $n$. Thus by choosing $k\geq n+3$, $\mathsf{Coll(}\Sigma_{n}) \in \mathsf{ZF}_{\Pi _{k}}$, which in turn makes it clear that $\mathsf{ZF}_{\Pi_n}$ is not solid.\footnote{Indeed, here we can even choose $k\geq n+2$. This is because $\mathsf{Coll}(\Sigma_{n+1})$ is equivalent to $\mathsf{Coll}(\Pi_n)$ (using the trick of collapsing consecutive existential quantifiers), and as verified in the proof of \cite[Lemma 3.1] {McKenzie-Kaufmann-Clote}, $\mathsf{Coll}(\Pi_{n})$ can be expressed as a collection of $\Pi_{n+2}$-sentences.} To see that $\mathsf{ZF}_{\Pi_n}$ is not tight, we use a similar reasoning, using   part (b) of Theorem \ref{N_and_K_n_are_biint_for_ZF+V=L}. This concludes the proof of Theorem \ref{Failure_of_tightness_for_bounded_ZF}.
  \end{proof} 
\begin{uwaga} For the purposes of Theorem \ref{Failure_of_tightness_for_bounded_ZF}, we could have stated Theorem \ref{Set_theoretical_analogue_of_Paris_Kirby} for models of $\mathsf{ZF+V=L}$. The reason we opted to state the theorem for models of $\mathsf{ZF^-+V=L}$ is that this more general version is needed to establish Theorem 38 of \cite{Enayat_Łełyk_2024}. 
\end{uwaga}
  \section{Correction to Theorem 77 of \cite{Enayat_Łełyk_2024}}

  Theorem 77 of \cite{Enayat_Łełyk_2024} states that the scheme template $\tau_{\mathsf{Repl+Tarski}}$  is both g-solid and strongly internally categorical. 
  \begin{itemize}
      \item We withdraw our claim to the veracity of Theorem 77 of \cite{Enayat_Łełyk_2024} since one of the lemmas employed in its proof (Lemma 79 of \cite{Enayat_Łełyk_2024}), as we shall explain, has turned out to be false. At the time of this writing we have neither an alternative proof of Theorem 77 of \cite{Enayat_Łełyk_2024}, nor a counterexample to demonstrate its falsity.
      \item In \cite{Piotr+Mateusz2025} it is verified that the G\"odel-Bernays class theory ${\sf GB}$ proves that any two full models of $\tau_{\mathsf{Repl+Tarski}}$ are isomorphic, which according to the definitions from \cite{Piotr+Mateusz2025} amounts to establishing that $\tau_{\mathsf{Repl+Tarski}}$ is categorical in $\tau_{\mathsf{Repl}}$. We stress that the problem specified in the above bullet item is avoided by assuming the replacement axiom for all class functions in the metatheory (${\sf GB}$).
  \end{itemize}

  
  The following is the statement of the aforementioned problematic lemma:

  \begin{claim}[Incorrect] \label{problematic_lemma} Suppose $\mathcal{M}$ and $\mathcal{K}$ are models of $\mathsf{ZF}$ such that the universe $K$ of $\mathcal{K}$ is a subset of the universe $M$ of $\mathcal{M}$ (and therefore the membership relation $\in^{\mathcal{K}}$ of $\mathcal{K}$ is a subset of $M^2$). Moreover, assume that for every parametrically $\mathcal{M}$-definable subset $D$ of $M$, $(\mathcal{K},D\cap K)\models \mathsf{ZF}(P),$ where $P$
is a fresh unary predicate interpreted by $D\cap K$. \textbf{Then} precisely one of the following three conditions hold:
\begin{enumerate}
    \item [(1)] There is an $
(\mathcal{M}, \mathcal{K})$-definable isomorphism between $\mathcal{K}$ and $\left(
V_{\kappa },\in \right) ^{\mathcal{M}}$ for some strongly inaccessible cardinal $%
\kappa $ of $\mathcal{M}$.

\item [(2)] There is an $(\mathcal{M}, \mathcal{K})$-definable isomorphism between $\mathcal{K}$ and $\mathcal{M}$.

\item [(3)] There is an $
(\mathcal{M}, \mathcal{K})$-definable
embedding of $\mathcal{M}$ as a proper rank initial segment of $\mathcal{K}$.
\end{enumerate}
  
  \end{claim}

We now present a counterexample to Claim \ref{problematic_lemma}. Let $\mathcal{K}$ be an $\omega$-standard model of $\mathsf{ZFC}$, and let $U$ be an element of $\mathcal{K}$ such that the statement ``$U$ is a nonprincipal ultrafilter over $\mathcal{P}(\omega)$" holds in $\mathcal{K}$. Let $\M$ be the internal ultrapower of $\mathcal{K}$ modulo $U$ whose elements are of the form $[f]_{U}$ for some $f$ in $\mathcal{K}$ such that $\mathcal{K}\models f:\omega \rightarrow V$, and  $[f]_{U}$ is the $U$-equivalence class of $f$. It is well-known that the following analogue of the \L o\'{s} Theorem for ultraproducts \cite{Chang+Keisler} holds in this context (thanks to the availability of the axiom of choice in $\mathcal{K}$).

\begin{tw}\label{thm_ultraproducts}For all $m$\textit{-ary formulae} 
$\varphi (x_{0},\cdot \cdot \cdot ,x_{m-1})$ in the first order language of set theory, and all $m$-tuples $\left\langle f_{i}:i<m\right\rangle $, where  $\mathcal{K}\models f_i:\omega \rightarrow V$ for each $i<m$, the following holds:

\begin{center}
$\mathcal{M}\models \varphi ([f_{0}]_{U},\cdot \cdot \cdot
,[f_{m-1}]_{U})$~~~iff~~~$\mathcal{K}\models\bigg[\left\{ n\in \omega:  \varphi \left( f_{0}(n ),\cdot \cdot \cdot
,f_{m-1}(n )\right) \right\} \in U\bigg]$.\smallskip
\end{center}
\end{tw}
\noindent For each $k\in\mathcal{K}$ let $\underline{k}$ be the constant function in $\mathcal{K}$ whose domain is $\omega$ and whose constant value is $k$. The above theorem assures us that the map $j:\mathcal{K}%
\rightarrow \mathcal{M}$ given by $j(k)=[\underline{k}]_{U}$ is an
elementary embedding, thus by identifying $\mathcal{K}$\ with the image of $%
j $, we can construe $\mathcal{M}$ as an elementary extension of $\mathcal{K}$.\footnote{Using Scott's trick, $\mathcal{M}$ is parametrically interpretable in $\mathcal{K}$ via an identity-preserving interpretation, but this fact does not play a role in the proof here.}

Let $\mathrm{id}$ be the identity map from $\omega$ to itself (in the sense of $\mathcal{K}$). Since $U$ was chosen to be nonprincipal, every member of $U$ is seen by $\mathcal{K}$ to be an unbounded subset of $\omega$. So by Theorem 7, $[\mathrm{id}]_{U}$ is above all the standard natural numbers and thus $\mathcal{M}$ is $\omega$-nonstandard. It is also easy to see, using Theorem \ref{thm_ultraproducts}, and the fact that $U$ is a member of $\mathcal{K}$, that $\mathcal{K}$ is a conservative elementary extension of $\mathcal{M}$, i.e., the intersection of every parametrically definable subset of $\M$ with $\mathcal{K}$ is parametrically definable in $\mathcal{K}$.  This latter property assures us that $\mathcal{K}$ and $\mathcal{M}$ satisfy the hypothesis of Claim \ref{problematic_lemma}. However, since $\mathcal{K}$ is $\omega$-standard and $\mathcal{M}$ is $\omega$-nonstandard, none of the three conditions of Claim \ref{problematic_lemma} are satisfied.

  \section{Recent advances}

  In this section, we briefly bring attention to recent progress on topics explored in \cite{Enayat_Łełyk_2024}. 
  \begin{enumerate}

 \item  In Theorem 16 of \cite{Enayat_Łełyk_2024}, we showed that two theories  $\mathsf{Z}_2 + \Pi^1_{\infty}$-$\mathsf{AC}$ and $\mathsf{ZF^{-}}+ \forall x~\left\vert {x}\right\vert  \leq \aleph_0$ are not definitionally equivalent (although they are bi-interpretable). After the publication of our paper, an alternative proof of this fact was given by Meadows and Chen \cite{Meadows+Chen_teasing_apart} that has the advantage of being conceptually simpler, and contrary to our proof, it does not rely on the consistency of the existence of an inaccessible cardinal with $\mathsf{ZFC}$. Their proof is based on the insight that every model of $\mathsf{Z}_2$ carries a definable global linear order, in contrast to $\mathsf{ZF^{-}}+ \forall x~\left\vert {x}\right\vert  \leq \aleph_0$, which is shown in \cite{Meadows+Chen_teasing_apart} to have models (obtained by forcing) that lack a definable global linear order. 
 

 \item The recent manuscript \cite{Piotr+Mateusz+Leszek2025} of Gruza, Kołodziejczyk, and Łełyk contains a number of results in which various proper subtheories of $\PA$ are constructed that exemplify the separation of various categoricity-like notions. The paper also establishes that for any fixed $n\in\omega$, there is a solid subtheory of $\PA$ that extends $\mathrm{I}\Sigma_n$ in which $\PA$ is not interpretable, thereby providing a strong positive answer to  Question B of \cite{Enayat_Łełyk_2024} for $T=\PA.$
 
 \item The recent manuscript \cite{Piotr+Mateusz2025} of Gruza and Łełyk further develops the work in \cite{Enayat_Łełyk_2024} concerning various `flavors' of internal categoricity. The following summarizes the highlights of the manuscript: 
 \begin{enumerate}
     \item Systematic study of the dependence of the categoricity results on the metatheory, and on the inclusion of parameters in the schematic presentations of theories.
     \item Formalization of the notion of "categoricity up to an ordinal height", and establishing that the means which are sufficient to establish the categoricity of second-, third-, fourth- ... order arithmetic are not sufficient to establish the categoricity of $\mathsf{ZF}$ up to an ordinal height.
     \item Establishing that if $\tau$ is a strongly internally categorical scheme (in the sense of \cite{Enayat_Łełyk_2024}), then $\tau[\mathcal{L}_{\tau}]$ (i.e. the theory consisting of all instantiations of $\tau$ with $\mathcal{L}_{\tau}$-formulae) is solid.
 \end{enumerate}

 \item In recent work, Meadows \cite{Meadows_int_cat_Steel} has explored the internal categoricity of Steel's theory of the multiverse. 

 \item The categorical-like notion of tightness has come to play a role in framing foundational issues about set theory, as exemplified in recent work of Barton \cite{Barton_metasemantics}, and Meadows \cite{Meadows_Sheep2025}.

 \item Elliot Glazer \cite{On_ZCR} has sketched out a proof of solidity of the fragment $\mathsf{ZCR}$ of $\mathsf{ZF}$, obtained by adding the axiom of ranks (denoted $\mathsf{R}$ here) to $\mathsf{ZC}$, where $\mathsf{ZC}$ is Zermelo set theory plus the axiom of choice. The axiom of ranks states that the universe can be written as the union of sets of the form $V_{\alpha}$, as $\alpha$ ranges over the ordinals. It remains open whether $\mathsf{ZR}$ is a solid theory.
\item The first-named-author has recently posted \cite{Enayat_Visserian_corrections} in relation to the solidity of the Kelley-Morse theory of classes (and its higher order variants).

  \end{enumerate}

\bibliographystyle{plain}
\bibliography{sonata.bib}

\end{document}